\documentclass[12 pt]{article}
\usepackage[dvips]{graphicx}
\usepackage{amssymb,amsmath,amsthm}

\newtheorem{definition}{Definition}

\newtheorem{theorem}{Theorem}
\newtheorem{guess}{Conjecture}
\newtheorem{lemma}{Lemma}

\newtheorem{corollary}{Corollary}
\newtheorem{proposition}{Proposition}

\begin{document}

\setlength{\oddsidemargin}{-.5in}
\setlength{\evensidemargin}{-.5in}

\title{Knots in the Canonical Book Representation of Complete Graphs}
\author{Andrea Politano and Dana Rowland}
\date{June 16, 2011}

\maketitle

\abstract{
We examine $\widetilde{K}_n$, the canonical book representation of the complete graph $K_n$, and describe knots that can be obtained as cycles in this particular spatial embedding.  We prove that for each knotted Hamiltonian cycle $\alpha$ in $\widetilde{K}_n$, there are at least $2^k \binom{n+k}{k}$ Hamiltonian cycles that are ambient isotopic to $\alpha$ in $\widetilde{K}_{n+k}$.  We show that when $p$ and $q$ are relatively prime with $p<q$, the $(p,q)$ torus knot is a Hamiltonian cycle in $\widetilde{K}_{2p+q}$.

We also show that the canonical book representation of $K_n$ contains a Hamiltonian cycle that is a composite knot if and only if $n\geq 12$.  We prove that if $\alpha$ is a knotted cycle in the canonical book representation of $K_n$ and $\beta$ is a knotted cycle in the canonical book representation of $K_m$, then there is a Hamiltonian cycle in $K_{n+m+1}$ that is ambient isotopic to a composite knot $\alpha \# \beta.$   Finally, we list the number and type of all non-trivial knots that occur as cycles in the canonical book representation of $K_n$ for $n\leq 11$.
}

\section{Introduction}
In $K_n$, the \emph{complete graph} on $n$ vertices, every pair of distinct vertices is joined by an edge.  An embedding or spatial representation of $K_n$ is a particular way of joining the $n$ vertices in three-dimensional space. In \cite{CG}, Conway and Gordon proved that every spatial representation of $K_6$ contains at least one pair of linked triangles and every spatial representation of $K_7$ contains at least one knotted Hamiltonian cycle.  They included examples of embeddings of $K_6$ and $K_7$ that were minimally linked or knotted--their embedding of $K_6$ contained exactly one pair of linked triangles and their embedding of $K_7$ contained exactly one knotted Hamiltonian cycle.

In \cite{otsuki}, Otsuki introduced a family of spatial representations of $K_n$ that generalized these examples of Conway and Gordon. Otsuki's spatial representation of $K_n$ is an example of a book representation. Projections of book representations prevent complicated interactions between edges.  In particular,
\begin{itemize}
\item No edge crosses itself.
\item A pair of edges cross at most once.
\item If edge $e_1$ crosses over edge $e_2$ and edge $e_2$ crosses over edge $e_3$, then edge $e_1$ crosses over edge $e_3$.
\end{itemize}
Because book representations minimize the entanglement among the edges in a graph, they are good candidates for minimizing the linking and knotting in an embedding of a graph.

Otsuki called his family of embeddings the \emph{canonical book representations} of $K_n$, which in this paper we denote as $\widetilde{K}_n$.  He showed that any sub-collection of $m$ vertices of $\widetilde{K}_n$ induces a subgraph that is ambient isotopic to $\widetilde{K}_m$.  Note this implies that for $n \geq 6$, $\widetilde{K}_n$ contains exactly $\binom{n}{6}$ linked triangles, all of which are ambient isotopic to the Hopf link, and for $n \geq 7$, $\widetilde{K}_n$ contains exactly $\binom{n}{7}$ knotted 7-cycles, all of which are trefoil knots.  Since Conway and Gordon's theorem implies that any spatial embedding of $K_n$ contains at least $\binom{n}{6}$ linked triangles and at least $\binom{n}{7}$ non-trivially knotted 7-cycles, a canonical book representation is minimally linked and knotted in this sense.

In addition, Fleming and Mellor proved that a canonical book representation of $K_n$ attains $14\binom{n}{7}$ triangle-square links, and showed this is the minimum possible for any embedding of $K_n$ \cite{fleming}.  They also conjectured that for any graph $G$ there is some book representation that realizes the minimal number of non-trivial links possible in an embedding of $G$.

Similarly, the canonical book representation $\widetilde{K}_n$ is a candidate for the minimal number of knotted cycles in an embedding.

In this paper, we focus on which knots arise as knotted Hamiltonian cycles in the canonical book representation for $n>7$. In section \ref{Canonical Book Representation} we review the definitions of book representations and Otsuki's canonical book representation, and show how knotted cycles in $\widetilde{K}_n$ are related to knotted cycles in $\widetilde{K}_{n+1}$.  In section \ref{Torus Knots} we show that $\widetilde{K}_n$ contains a $(p,q)$ torus knot (or link) when $n\geq q+2p$.  In section \ref{Composite Knots} we examine composite knots in the canonical book representation, and in section \ref{Conclusion} we give a listing of all the knots that appear as cycles in $\widetilde{K}_n$ for $8\leq n \leq 11$ and we conjecture about the ways in which $\widetilde{K}_n$ may achieve the minimal possible knotting complexity.

\section{The Canonical Book Representation of $K_n$}
\label{Canonical Book Representation}
In this section, we review the right canonical book representation, as defined in \cite{otsuki}.  (In the right canonical book representation, the knotted 7-cycles are right-handed trefoil knots.  The left canonical book representation is the mirror image of the one presented here.)

\begin{definition}
A \emph{$k$-book} is a subset of $\mathbb{R}^3$ consisting of a line $L$ and distinct half-planes $S_1$,$S_2$,...,$S_k$ with boundary $L$.  The line $L$ forms the spine of the book and the half-planes $S_i$ form the pages, or sheets.  We denote a $k$-book by $B_k$.  Let $G$ be a graph, and let $f: G\rightarrow B_k \subset \mathbb{R}^3$ be a tame embedding of $G$.  We say that the spatial representation $f(G)$ is a $k$-book representation of $G$ if:
\begin{enumerate}
\item Each vertex of $f(G)$ is on the line $L$
\item Each edge of $f(G)$ is contained in exactly one sheet $S_i$.
\end{enumerate}
\end{definition}

If $\widetilde{G}$ is a $k$-book representation of $G$, then $\widetilde{G}$ can be deformed by an ambient isotopy so that the vertices lie on a circle $C$ and the edges are chords on $k$ internally disjoint topological disks, all of which have $C$ as their boundary.  For the remainder of this paper, we will treat the sheets $S_i$ for $1\leq i\leq k$ as topological disks.  In a projection of the embedding onto the plane containing $C$, we assume that the sheets are labeled so that sheet $S_i$ is ``above'' sheet $S_j$ if $i<j$.  A $k$-book embedding is determined, up to ambient isotopy, by specifying which edges are in which sheet.

The \emph{sheet-number} of a graph $G$ is the smallest possible $k$ for which $G$ has a $k$-book representation.

For $K_n$, the sheet-number is $\lceil{n/2}\rceil,$ the smallest integer greater than or equal to $n/2$; see \cite{Bernhart} or \cite{Kobayashi} for proofs.  Otsuki's right canonical book representation, $\widetilde{K}_n$,  provides an example of a minimal-sheet book embedding of $K_n$ \cite{otsuki}.

To describe $\widetilde{K}_n$, it suffices to list which sheet contains each edge.  Label the $n$ vertices with the integers 1 through $n$.

\textbf{Case 1: The number of vertices is even.}

When $n=2m$, there are $m$ sheets and each of the sheets $S_1$, $S_2$, ..., $S_m$ contains $2m-1$ edges.  Sheet $S_i$ contains the edges joining vertex $i$ to vertex $i+j$, for $1 \leq j \leq m$, and the edges joining vertex $i+m$ to $(i+m+j)$ mod $2m$, for $1 \leq j \leq m-1$.  See Figure \ref{even}.

\begin{figure}
\centering
\includegraphics{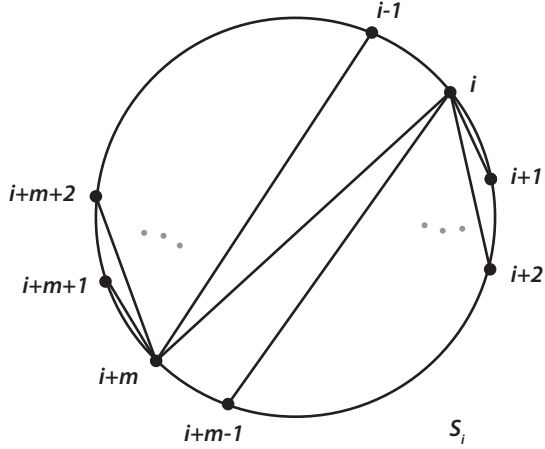}
\caption{Sheet $S_i$ in the canonical book representation of $K_{2m}$}
\label{even}
\end{figure}

Alternatively, if we are given an edge joining vertex $i$ to vertex $j$, we can determine which sheet the edge is in:

\begin{lemma}
Let $n=2m$ and let $(i,j)$ be the edge joining vertex $i$ to vertex $j$ in the projection of $\widetilde{K}_n$.  Assume that $i<j$.  Then we can determine which sheet contains $(i,j)$.
\begin{itemize}
\item   If $i \leq m$ and $j-i \leq m$, then $(i,j)$ is in $S_i$.
\item   If $i \leq m$ and $j-i \geq m+1$, then $(i,j)$ is in $S_{j-m}$.
\item	If $i \geq m+1$, then $(i,j)$ is in $S_{i-m}$.
\end{itemize}
Furthermore, suppose the edge $(i,j)$ crosses the edge $(k,l)$ in the projection.  We may assume without loss of generality that $1\leq i<k<j<l \leq 2m$.  Then edge $(k,l)$ is on top of edge $(i,j)$ if and only if $i \leq m$ and $k \geq m+1$.
\label{evenCriterion}
\end{lemma}

\textbf{Case 2:  The number of vertices is odd.}

When $n=2m+1$, the sheets $S_1$, $S_2$, ..., $S_m$ each contain $2m$ edges and sheet $S_{m+1}$ is a ``half-sheet'' containing $m$ edges.  For each $1\leq i \leq m$, sheet $S_i$ contains the edges joining vertex $i$ to vertex $i+j$, and the edges joining vertex $i+m+1$ to $(i+m+j+1)$ mod $(2m+1)$, for $1 \leq j \leq m$.  Sheet $S_{m+1}$ contains the edges joining vertex $m+1$ to vertex $m+1+j$, for $1 \leq j \leq m$.  See Figure \ref{odd}.

\begin{figure}
\centering
\includegraphics{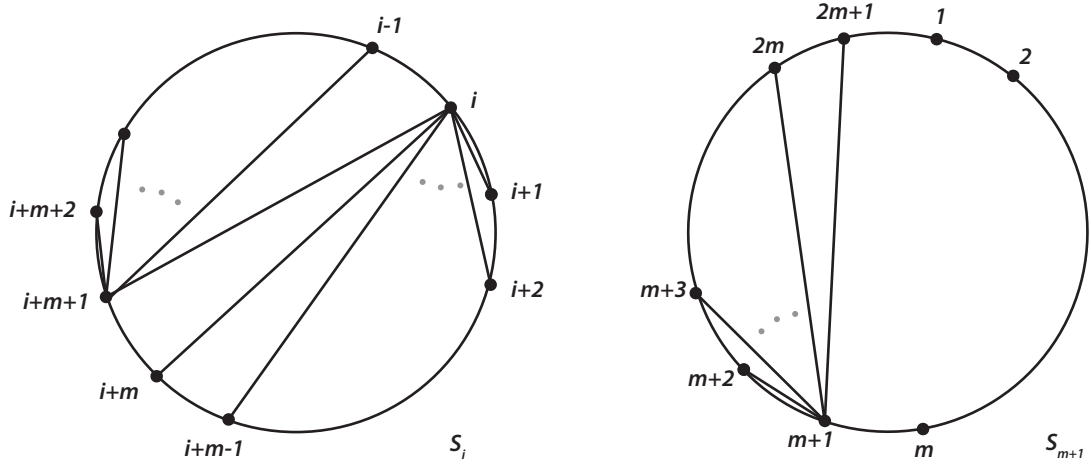}
\caption{Sheets $S_i$ ($i\leq m$) and $S_{m+1}$ in the canonical book representation of $K_{2m+1}$}
\label{odd}
\end{figure}

If we are given an edge joining vertex $i$ to vertex $j$, we can determine which sheet the edge is in:

\begin{lemma}
Let $n=2m+1$, and let $(i,j)$ be the edge joining vertex $i$ to vertex $j$ in the projection of $\widetilde{K}_n$.  Assume that $i<j$.  Then we can determine which sheet contains $(i,j)$.
\begin{itemize}
\item   If $i \leq m+1$ and $j-i \leq m+1$, then the edge is in $S_i$.
\item   If $i \leq m+1$ and $j-i \geq m+2$, then the edge is in $S_{j-m-1}$.
\item	If $i \geq m+2$, then the edge is in $S_{i-m-1}$.
\end{itemize}
Furthermore, suppose the edge $(i,j)$ crosses the edge $(k,l)$ in the projection.  We may assume without loss of generality that $1\leq i<k<j<l \leq 2m+1$.  Then edge $(k,l)$ is on top of edge $(i,j)$ if and only if $i \leq m+1$ and $k \geq m+2$.
\label{oddCriterion}
\end{lemma}

Otsuki proved that the canonical book representation has the property that any subgraph induced by a sub-collection of vertices is ambient isotopic to a canonical book representation \cite{otsuki}.  In particular, we have the following:
\begin{proposition}
Let $\alpha=(\alpha_1, \alpha_2, ..., \alpha_n)$ be an $n$-cycle through the vertices $1, ... ,n$ in $\widetilde{K}_{N}$.  Then the $n$-cycle $(\alpha_1, \alpha_2, ..., \alpha_n)$ through the vertices $1, ... ,n$ in $\widetilde{K}_{N+1}$ is ambient isotopic to $\alpha$.
\label{addVertices}
\end{proposition}
\begin{proof}

Let $(i,j)$ and $(k,l)$ be edges of the cycle $\alpha$ in $\widetilde{K}_{N}$, labeled such that $i<j$, $k<l$, and $i<k$.  Two edges cross in the projection of $\widetilde{K}_{N}$ if and only if they cross in the projection of $\widetilde{K}_{N+1}$, which occurs if and only if $i<k<j<l$.

First, consider the case $N=2m+1$.  We can use Lemmas \ref{evenCriterion} and \ref{oddCriterion} to verify that:
\begin{enumerate}
\item If $i<k \leq m+1$ or if $m+2 \leq i <k$, then $(i,j)$ crosses over $(k,l)$ in both $\widetilde{K}_N$ and $\widetilde{K}_{N+1}$
\item If $i \leq m+1$ and $k \geq m+2$, then $(k,l)$ crosses over $(i,j)$ in both $\widetilde{K}_N$ and $\widetilde{K}_{N+1}$
\end{enumerate}
Since there are no crossing changes between edges, the cycle $(\alpha_1, \alpha_2, ..., \alpha_n)$ represents the same knot in both $\widetilde{K}_{2m+1}$ and $\widetilde{K}_{2m+2}$.

Now suppose that $N=2m$.
Using Lemmas \ref{evenCriterion} and \ref{oddCriterion} we observe that:
\begin{enumerate}
\item If $i<k \leq m$ or if $m+2 \leq i <k$, then $(i,j)$ crosses over $(k,l)$ in both $\widetilde{K}_N$ and $\widetilde{K}_{N+1}$
\item If $i \leq m$ and $k \geq m+2$, then $(k,l)$ crosses over $(i,j)$ in both $\widetilde{K}_N$ and $\widetilde{K}_{N+1}$
\item If $i=m+1$ or if $k=m+1$ then a crossing change occurs between edges $(i,j)$ and $(k,l)$ when moving from $\widetilde{K}_N$ to $\widetilde{K}_{N+1}$
\end{enumerate}

Notice that if $i=m+1$ (or $k=m+1$), then $(i,j)$ (or respectively $(k,l)$) is in the top sheet in $\widetilde{K}_N$ and the bottom sheet in $\widetilde{K}_{N+1}$.  An edge in the bottom sheet of $\widetilde{K}_{N+1}$ is under \emph{all} other edges and can be moved by an ambient isotopy so that it lies over all other edges.  Thus, the only crossing changes that occur do not change the knot type, and the cycle $(\alpha_1, \alpha_2, ..., \alpha_n)$ represents the same knot in both $\widetilde{K}_{2m}$ and $\widetilde{K}_{2m+1}$.

\end{proof}

We also know that if a Hamiltonian cycle with a certain knot type appears in $\widetilde{K}_n$, then $\widetilde{K}_N$ must contain a Hamiltonian cycle with the same knot type for any $N>n$.  The following theorem indicates one way to find such a cycle:

\begin{theorem}
Let $\alpha=(\alpha_1, \alpha_2, ..., \alpha_n)$ be an $n$-cycle through all the vertices $1, 2, ..., n+1$ except $i+1$ in $\widetilde{K}_{n+1}$.  Suppose that $\alpha_k=i$ and $\alpha_{k+1}=j$.
Then the Hamiltonian cycle $(\alpha_1, ..., \alpha_k, i+1, \alpha_{k+1}, ..., \alpha_n)$ is ambient isotopic to $\alpha$.
\label{extendCycle}
\end{theorem}

\begin{proof}
It suffices to check that in $\widetilde{K}_n$ any edge $(i,j)$ is at most one sheet apart from the edge $(i+1,j)$.  This will guarantee that the edge $(i,j)$ can be moved to the path $(i, i+1, j)$ by an ambient isotopy, since if the edge $(i+1,j)$ is one sheet level above or one below the edge $(i,j)$ then the path $(i,i+1,j)$ crosses the same edges as the edge $(i,j)$ and in the same manner.  In other words, no edge can pass through the triangle formed by the cycle $(i,i+1,j)$.  Note that the top and bottom sheets can also be considered consecutive, since an edge on the bottom sheet can be deformed by ambient isotopy to be on top of all the sheets, and vice versa.

We will verify that edges $(i,j)$ and $(i+1,j)$ are at most one sheet apart when $i<j$.  The proof for when $i>j$ is similar, and is left to the reader.  There are six cases to check.

\textbf{Case 1: $i\geq m+1$ and there are an even number of vertices.}
Refer to Lemma \ref{evenCriterion}.  The edge $(i,j)$ is in $S_{i-m}$. The  edge $(i+1,j)$ is in $S_{i+1-m}$. Therefore, the edges are in consecutive sheets.

\textbf{Case 2: $i\geq m+2$ and there are an odd number of vertices.}
Refer to Lemma \ref{oddCriterion}.  The edge $(i,j)$ is in $S_{i-m-1}$.  The edge $(i+1,j)$ is in $S_{i+1-m-1}$ which equals $S_{i-m}$. Therefore, the edges are in consecutive sheets.

\textbf{Case 3: $i\leq m$, $j-i\leq m$ and there are and even number of vertices.}
Refer to Lemma \ref{evenCriterion}.  The edge $(i,j)$ is in $S_i$. There are two possibilities for the sheet level of the edge $(i+1,j)$. First, if $i<m$, then $i+1\leq m$, and $j-(i+1) = j - i - 1\leq m-1 \leq m$. Therefore, edge $(i+1,j)$ is in $S_{i+1}$. Second, if $i=m$, then $i+1\geq m+1$ so edge $(i+1,j)$ is in $S_{i+1-m} = S_{m+1-m} = S_1$. This would not change the knot type because the edge $(i,j)$ was in the very last sheet and this edge is in the very first sheet. In both cases, the edges are in consecutive sheets.

\textbf{Case 4: $i\leq m+1$, $j-i\leq m+1$ and there are an odd number of vertices.}
Refer to Lemma \ref{oddCriterion}.  The edge $(i,j)$ is in $S_i$.  Again, there are two possibilities for the sheet level of edge $(i+1,j)$. First, if $i<m+1$, then $i+1\leq m+1$, and so $j-(i+1) = j-i-1\leq m\leq m+1$ meaning this edge is found in $S_{i+1}$. This is one sheet level below the original edge. Second, if $i=m+1$, then $i+1\geq m+2$. The edge $(i+1,j)$ is therefore in $S_{i+1-m-1}=S_{m+1+1-m-1}=S_1$, the very first sheet. As in case 3, this means that the knot type remains unchanged.

\textbf{Case 5: $i\leq m$, $j-i\geq m+1$ and there are an even number of vertices.}
Refer to Lemma \ref{evenCriterion}.  The edge $(i,j)$ is in $S_{j-m}$. Note that if $i=m$, then $j\geq 2m+1$ which is impossible, since there are only $2m$ vertices.  That leaves two possibilities for the sheet level of edge $(i+1,j)$. First, if $i<m$ and $j>i+m+1$, then $i+1\leq m$ and $j-(i+1)\geq m+1$. This forces the edge to be in $S_{j-m}$, and so both edges are in the same sheet. Second, if $i<m$ and $j=i+m+1$, then $i+1\leq m$ and $j-(i+1) = m$. This means the edge is in $S_{i+1}$, which is equivalent to $S_{j-m}$ because $j=i+m+1$. Again, both edges are in the same sheet.

\textbf{Case 6: $i\leq m+1$, $j-i\geq m+2$ and there are an odd number of vertices.}
Refer to Lemma \ref{oddCriterion}.  The edge $(i,j)$ is in $S_{j-m-1}$.  Note that if $i=m+1$, then $j \geq 2m+3$ which is impossible, so we can assume $i < m+1$.  There are two possibilities for the sheet level of edge $(i+1,j)$. First, if $i<m+1$ and $j>i+m+2$, then $i+1\leq m+1$ and $j-(i+1)\geq m+2$. This means the edge is in $S_{j-m-1}$. Second, if $i<m+1$ and $j=i+m+2$, then $i+1\leq m+1$ and $j-(i+1)= m+1$. Again, the edge is in $S_{i+1}=S_{j-m-1}$. Both edges are in the same sheet.
\end{proof}

\begin{corollary}
Suppose $\alpha$ is a Hamiltonian cycle in $\widetilde{K}_{n}$ with the property that no edge of $\alpha$ joins consecutively labeled vertices.  Let $N=n+k$ for $k\geq 0$.  Then $\widetilde{K}_N$ contains at least $2^k \binom{N}{k}$ Hamiltonian cycles that are ambient isotopic to $\alpha$.
\end{corollary}

\begin{proof}
The subgraph induced by any $n$ vertices of $\widetilde{K}_N$ is ambient isotopic to $\widetilde{K}_n$, so there are at least $\binom{N}{n}$ $n$-cycles in $\widetilde{K}_N$ that are ambient isotopic to $\alpha$.  These cycles share the property that no edge joins consecutive vertices.  Let $(\alpha_1, \alpha_2, ..., \alpha_n)$ be such an $n$-cycle.  Choose the smallest integer $j$ such that $j =\alpha_i$ for some $1\leq i \leq n$ but $j+1\neq \alpha_l$ for any $1\leq l\leq n$.  (Note:  if $j=N$, we interpret $j+1$ as 1.)  By the proof of Theorem \ref{extendCycle}, the cycles $(\alpha_1, ..., \alpha_{i-1}, j+1, \alpha_i, ..., \alpha_n)$ and $(\alpha_1, ..., \alpha_{i}, j+1, \alpha_{i+1}, ..., \alpha_n)$ are both ambient isotopic to $\alpha$.  Repeat this step until all vertices in $\widetilde{K}_N$ are used.  This gives $2^k$ ways to extend each $n$-cycle, which produces $2^k \binom{N}{k}$ distinct Hamiltonian cycles that are ambient isotopic to $\alpha$, as claimed.
\end{proof}

This immediately implies that there are at least $2^{N-7} \cdot \binom{N}{7}$ Hamiltonian cycles that are trefoil knots in $\widetilde{K}_N$ when $N \geq 7$. This bound is not sharp, however, as shown in the table in Section \ref{Conclusion}.

\section{Torus knots in the Canonical Book Embedding}
\label{Torus Knots}
Recall that a \emph{torus link} is a knot or link that can be embedded on the standard (unknotted) torus in $\mathbb R^3$.  A $(p,q)$ torus link can be deformed so that it crosses every meridian (a closed curve that bounds a topological disk that is ``inside'' the torus) of the torus $p$ times and every longitude (a closed curve that bounds a topological disk that is ``outside'' the torus) of the torus $q$ times.  When $p$ and $q$ are relatively prime, the link is a knot.  See \cite[Section 5.1]{adams} for a general description of $(p,q)$ torus knots and links.

A $(p,q)$ torus knot can also be described as the closure of a braid on $p$ strands, with braid word $(\sigma_1\sigma_2 \ldots \sigma_{p-1})^q$. Recall that $\sigma_i$ denotes that the $i$th strand of the braid crosses over the $i+1$st strand of the braid, and equivalent braid words can be obtained using the braid relations $\sigma_i\sigma_{i+1}\sigma_i=\sigma_{i+1}\sigma_i\sigma_{i+1}$ and $\sigma_i \sigma_j = \sigma_j \sigma_i$ when $|i-j|\geq 2$.   See \cite[Section 5.4]{adams} or \cite{BraidsRef_2} for references on braids.

Consider the Hamiltonian cycle $(1,3,5,...,2m+1,2,4,...,2m)$ in $\widetilde{K}_{2m+1}$.  This cycle forms the closure of a 2-strand braid with $2m+1$ crossings.  See Figure \ref{cycleK9}.  For each $i \leq 2m-2$, we know from Lemma \ref{oddCriterion} that the edge $(i,i+2)$ crosses over the edge $(i+1, i+3)$ except when $i=m+1$.  Edge $(2m-1,2m+1)$ crosses over edge $(2m,1)$, edge $(2m,1)$ crosses over edge $(2m+1,2)$, and edge $(2m+1,2)$ crosses under edge $(1,3)$.  The resulting braid word is $\sigma^{m}\sigma^{-1}\sigma^{(2m-2)-(m+1)}\sigma\sigma\sigma^{-1}=\sigma^{2m-3}$.   Therefore, we see that $\widetilde{K}_{2m+1}$ contains a $(2,2m-3)$ torus knot.  When $q$ is odd, $\widetilde{K}_n$ contains a $(2,q)$ torus knot as one of its Hamiltonian cycles for all $n \geq q+4$.  (Note:  When $q$ is even, the same argument shows that $\widetilde{K}_n$ contains a $(2,q)$ torus link.)

\begin{figure}
\centering
\includegraphics{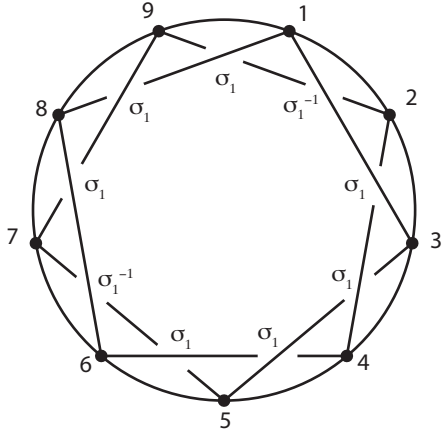}
\caption{The cycle $(1,3,5,7,9,2,4,6,8)$ in $\widetilde{K}_9$ is the knot $5_1$.  It can be described by the braid word $\sigma_1^4\sigma_1^{-1}\sigma_1^3\sigma_1^{-1}=\sigma_1^5$.}
\label{cycleK9}
\end{figure}

Suppose $n>6$ is not a multiple of 3, and consider the Hamiltonian cycle $(1,4,7,...)$ in $\widetilde{K}_n$.  This cycle forms the closure of a 3-strand braid with word $\prod_{i=1}^n \left( \sigma_1^{\delta_1(i)}\sigma_2^{\delta_2(i)} \right)$ where $\delta_1(i) = 1$ if the edge $(i,i+3)$ is over the edge $(i+1,i+4)$ and $-1$ otherwise, and $\delta_2(i) = 1$ if the edge $(i,i+3)$ is over the edge $(i+2,i+5)$, and $-1$ otherwise.  (The vertex labels are to be taken modulo $n$.)

Suppose $n=2m$ is even.  (The case for when $n$ is odd is similar, and omitted.)  Then Lemma \ref{evenCriterion} implies that $\delta_1(i) = -1$ if and only if $i$ is $m$ or $n$, and $\delta_2(i) = -1$ if and only if $i$ is one of $m-1, m, n-1$ or $n$.  The braid word becomes \[(\sigma_1\sigma_2)^{m-2}\sigma_1\sigma_2^{-1}\sigma_1^{-1}\sigma_2^{-1}(\sigma_1\sigma_2)^{m-2}\sigma_1\sigma_2^{-1}\sigma_1^{-1}\sigma_2^{-1}.\] Since the braid relations imply that \[\sigma_2^{-1}\sigma_1^{-1}\sigma_2^{-1}=\sigma_1^{-1}\sigma_2^{-1}\sigma_1^{-1}\]  we see that $\sigma_1 \sigma_2^{-1} \sigma_1^{-1}\sigma_2^{-1}\sigma_1\sigma_2$ is the identity.  Therefore the braid word can be reduced to $(\sigma_1\sigma_2)^{n-6}$.  This shows that $\widetilde{K}_{n}$ contains a $(3,n-6)$ torus knot.  For any $n \geq q+6$, the spatial representation $\widetilde{K}_{n}$ contains a $(3,q)$ torus knot (or link, if $q$ is a multiple of 3).

An extension of this argument leads to the following theorem:

\begin{theorem}
Let $p$, $q$, and $n$ be positive integers such that $p \leq q$ and $n\geq q+2p$.  Then the canonical book representation of $K_n$ contains a $(p,q)$ torus knot (or link).
\end{theorem}

\begin{proof}
By Theorem \ref{extendCycle}, it suffices to prove this theorem when $n=2p+q$.  Consider the knot or link in $\widetilde{K}_{2p+q}$ consisting of all edges of the form $(i, i+p)$ for $1 \leq i \leq n$, where the vertex labels are taken modulo $n$.  This knot or link can be described as a braid on $p$ strands with braid word
\[ w=\prod_{i=1}^n \left[ \sigma_1^{\delta_1(i)}\sigma_2^{\delta_2(i)}...\sigma_{p-1}^{\delta_{p-1}(i)} \right]  \]
where
\[ \delta_j(i) =  \left\{ \begin{array}{rl}
                         1 & \mathrm{if\ edge\ } (i,i+p) \mathrm{\ is\ over\ edge\ } (i+j,i+j+p) \\
                         -1 & \mathrm{otherwise}
                       \end{array}
  \right.
\]

We will use Lemma \ref{evenCriterion} to prove the case when $n$ is even.  The case for $n$ odd is left to the reader.

Suppose $n=2m$.  By Lemma \ref{evenCriterion},
\[
\delta_j(i) = \left\{ \begin{array}{rl}
                          1 & \mathrm{if\ }1 \leq i \leq m-j \\
                         -1 & \mathrm{if\ }m-j+1 \leq i \leq m \\
                          1 & \mathrm{if\ }m+1 \leq i \leq n-j \\
                         -1 & \mathrm{if\ }n-j+1 \leq i \leq n
                       \end{array}
                        \right.
\]
and this allows us to express $w$ as
\[ w=\left[
    (\sigma_1\sigma_2\ldots\sigma_{p-1})^{m-p+1}(\sigma_1\sigma_2\ldots\sigma_{p-2})\sigma_{p-1}^{-1}(\sigma_1\sigma_2\ldots\sigma_{p-3})\sigma_{p-2}^{-1}\sigma_{p-1}^{-1}\cdots\sigma_1^{-1}\sigma_2^{-1}\cdots\sigma_{p-1}^{-1}
\right]^2  .\]

Next, observe that \[(\sigma_1\sigma_2\ldots\sigma_{p-1})(\sigma_1\sigma_2\ldots\sigma_{p-2})\sigma_{p-1}^{-1}(\sigma_1\sigma_2\ldots\sigma_{p-3})\sigma_{p-2}^{-1}\sigma_{p-1}^{-1}\cdots\sigma_1^{-1}\sigma_2^{-1}\cdots\sigma_{p-1}^{-1} \]
is equivalent to the identity.   For example, when $p=4$, we can use the braid relations to obtain:

\begin{eqnarray*}
(\sigma_1\sigma_2\sigma_3)(\sigma_1\sigma_2\sigma_3^{-1})(\sigma_1\sigma_2^{-1}\sigma_3^{-1})(\sigma_1^{-1}\sigma_2^{-1}\sigma_3^{-1}) & = & \\
(\sigma_1\sigma_2\sigma_1)(\sigma_3\sigma_2\sigma_3^{-1})(\sigma_1\sigma_2^{-1}\sigma_1^{-1})(\sigma_3^{-1}\sigma_2^{-1}\sigma_3^{-1}) & = & \\
(\sigma_2\sigma_1\sigma_2)(\sigma_2^{-1}\sigma_3\sigma_2)(\sigma_2^{-1}\sigma_1^{-1}\sigma_2)(\sigma_2^{-1}\sigma_3^{-1}\sigma_2^{-1}) & = & \\
\sigma_2\sigma_1\sigma_3\sigma_1^{-1}\sigma_3^{-1}\sigma_2^{-1} & = & \\
\sigma_2\sigma_3\sigma_1\sigma_1^{-1}\sigma_3^{-1}\sigma_2^{-1} & = & 1
\end{eqnarray*}
See Figure \ref{braid}. 

\begin{figure}
\centering
\includegraphics{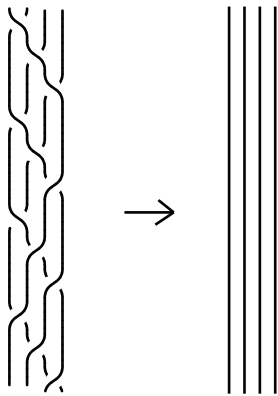}
\caption{
The braid word $(\sigma_1\sigma_2\sigma_3)(\sigma_1\sigma_2\sigma_3^{-1})(\sigma_1\sigma_2^{-1}\sigma_3^{-1})(\sigma_1^{-1}\sigma_2^{-1}\sigma_3^{-1})$ is equivalent to the identity.
}
\label{braid}
\end{figure}

This implies that the braid word simplifies to $w=(\sigma_1\sigma_2\ldots\sigma_{p-1})^{2m-2p}=(\sigma_1\sigma_2\ldots\sigma_{p-1})^q$, so we obtain a $(p,q)$ torus link as claimed.

\end{proof}

\section{Composite knots in the canonical book embedding}
\label{Composite Knots}
In this section we prove that the canonical book embedding of $K_n$ contains a composite knot for all $n\geq 12$. We also show that if we choose any two knotted Hamiltonian cycles contained in $\widetilde{K}_p$ and $\widetilde{K}_q$ respectively, their composite will be a Hamiltonian cycle in $\widetilde{K}_{p+q+1}$.

\begin{theorem}
Let $n \geq 14$.  Then the cycle
\[(1,3,5,7,9,11,13,8,10,12,14,15,16,...,n,2,4,6)\]
in the canonical book representation of $K_n$ is the composite of two trefoils.
\label{composite}
\end{theorem}

\begin{proof}
We can find a composite knot in $\widetilde{K}_{14}$ by first finding trefoils in two disjoint subgraphs.  The first subgraph is induced by vertices 1 through 7.  The second subgraph is induced by vertices 8 through 14.

Any set of seven vertices of $\widetilde{K}_{14}$ induces a graph that is ambient isotopic to the canonical book representation of $K_7$.  In $\widetilde{K}_7$ there is exactly one trefoil knot.  Therefore, there is exactly one trefoil in each subgraph of $\widetilde{K}_{14}$ induced by seven vertices. The first subgraph has a trefoil in the cycle $(1,3,5,7,2,4,6)$.  Notice that this cycle is ambient isotopic to the cycle $(1,3,5,7,14,2,4,6)$ in $\widetilde{K}_{14}$.  See Figure \ref{knot1}.

\begin{figure}
\centering
\includegraphics{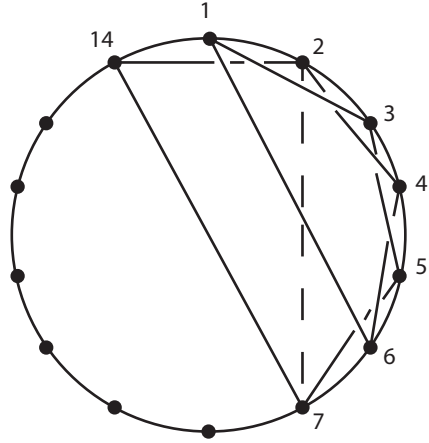}
\caption{Cycle (1,3,5,7,14,2,4,6) is ambient isotopic to cycle (1,3,5,7,2,4,6) since the edge (2,7), shown as a dashed line, can be replaced by the path (2,14,7).}
\label{knot1}
\end{figure}

These cycles are ambient isotopic because the only edges of the cycles which intersect the path (2,14,7) and the edge (2,7) are edges (1,3) and (1,6). Both of these edges lie in $S_1$ meaning that any path or edge that crosses those two edges will fall in a lower sheet. This means that the edge $(2,7)$ can be replaced with the path $(2,14,7)$ without changing the knot type.

The second subgraph (induced by vertices 8 through 14) has a trefoil in the cycle:
\[(8,10,12,14,9,11,13).\]
Notice that this cycle is ambient isotopic to the cycle:
\[(8,10,12,14,7,9,11,13).\]  See Figure \ref{knot2}.

\begin{figure}
\centering
\includegraphics{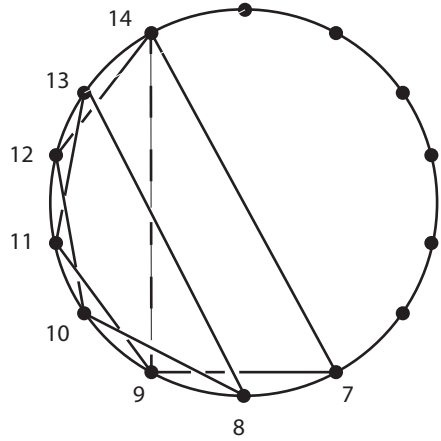}
\caption{Cycle (8,10,12,14,7,9,11,13) is ambient isotopic to cycle(8,10,12,14,9,11,13) since the edge (9,14), shown as a dashed line, can be replaced by the path (9,7,14).}
\label{knot2}
\end{figure}

These cycles are ambient isotopic because both the path (9,7,14) and the edge (9,14) cross edges (8,10) and (8,13), which are both in $S_1$. Therefore, any edge or path that crosses these two edges will still remain under them, meaning that the path $(9,7,14)$ can be replaced with the edge $(9,14)$ without affecting the knot type.

Place the two cycles on $\widetilde{K}_{14}$.  When these two cycles are layered they share the edge (7,14).  Removing edge (7,14)(which is the shared edge that has no crossings), will create the composite of the two trefoils. See Figure \ref{knot3}.  The cycle with the composite knot in $\widetilde{K}_{14}$ is:
\[(1,3,5,7,9,11,13,8,10,12,14,2,4,6).\]

\begin{figure}
\centering
\includegraphics{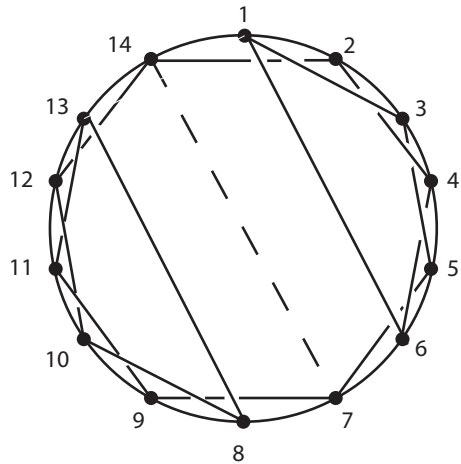}
\caption{Composite knot in $K_{14}$. The dashed line is the edge removed from both factor knots to form the composite.}
\label{knot3}
\end{figure}

For $n > 14$, the fact that the cycle
\[(1,3,5,7,9,11,13,8,10,12,14,15,16,...,n,2,4,6)\]
in the canonical book representation of $K_n$ is the composite of two trefoils follows immediately from Theorem 1.
\end{proof}

We can improve this result by finding a composite knot in $\widetilde{K}_{13}$. Consider two subgraphs of $\widetilde{K}_{13}$.  Let the first subgraph of $\widetilde{K}_{13}$ be induced by vertices 1 through 7, and let the second subgraph be induced by vertices 7 through 13. Refer to Figure \ref{knot4}.

Since each subgraph is ambient isotopic to $\widetilde{K}_7$, each subgraph contains exactly one trefoil knot. The first subgraph has a trefoil knot in the cycle:
\[(1,3,5,7,2,4,6).\]
The second subgraph has a trefoil in the cycle:
\[(7,9,11,13,8,10,12).\]

Place these two cycles together in $\widetilde{K}_{13}$. Notice that 4 edges meet at vertex 7.  Connect the knots by joining edges (5,7) and (7,9) and replacing the path (2,7,12) with the edge (2,12). This results in the cycle
\[(1,3,5,7,9,11,13,8,10,12,2,4,6).\]
Note that edge (2,12) crosses edges (1,3), (1,6), (8,13) and (11,13). Edge (2,12) is in sheet five, edges (1,3), (1,6), and (8,13) are in sheet one and lastly, edge (11,13) is in sheet four. This means that edge (2,12) crosses completely under all edges. Since edges (2,7) and (7,12) also cross under all the edges that edge (2,12) crosses, replacing the path (2,7,12) by the edge (2,12) forms a composite of the two trefoil knots in $\widetilde{K}_{13}$.

\begin{figure}
\centering
\includegraphics{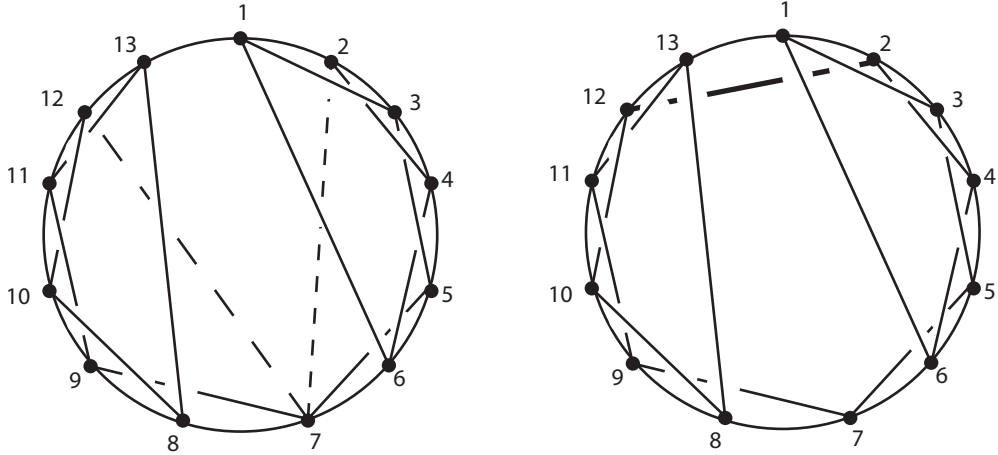}
\caption{Composite knot in $K_{13}$. On the left, the two cycles are layered with the dashed lines representing the edges to be removed. On the right is the composite formed by adding the bold edge (2,7).}
\label{knot4}
\end{figure}

The smallest $\widetilde{K}_n$ that a composite knot can be found in is $\widetilde{K}_{12}$; refer to Figure \ref{knot5}. To find this composite, once again we consider two subgraphs of $\widetilde{K}_{12}$ where the first subgraph is induced by the first 7 vertices and the second subgraph is induced by the last 7 vertices in the embedding of $K_{12}$.  

Each subgraph contains exactly one Hamiltonian cycle that is a trefoil knot. The first subgraph has a trefoil in the cycle:
\[(1,3,5,7,2,4,6).\]
The second subgraph has a trefoil in the cycle:
\[(6,8,10,12,7,9,11).\]

Place these two cycles with the trefoil knots on $K_{12}$. Notice that there are 4 edges that meet at vertex 6 and vertex 7.  Removing the paths (8,6,11) and (2,7,5) and adding edges (2,11) and (5,8) forms a composite knot. This cycle is:
\[(1,3,5,8,10,12,7,9,11,2,4,6).\]

\begin{figure}
\centering
\includegraphics{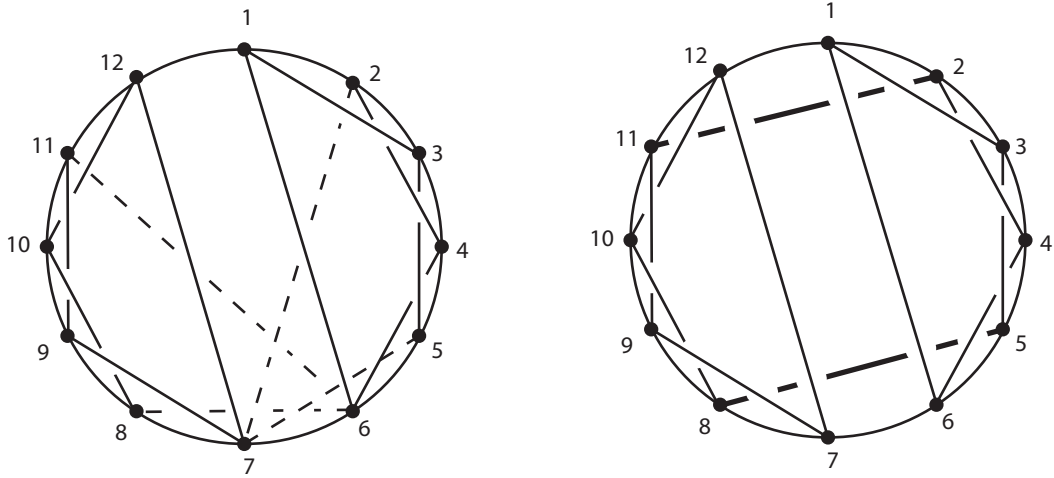}
\caption{Composite knot in $K_{12}$. On the left, two knotted trefoils are shown.  The dashed edges are the ones that will be replaced. On the right is a cycle which is the composite of the two trefoils.}
\label{knot5}
\end{figure}

Up to now we have shown how to find composites of trefoil knots.  A similar method can be used to find other composite knots.

\begin{theorem}
Let $\alpha$ be a Hamiltonian cycle in the canonical book representation of $K_p$ and let $\beta$ be a Hamiltonian cycle in the canonical book representation of $K_q$.  Then $\alpha \# \beta$ is a Hamiltonian cycle in the canonical book representation of $K_{p+q+1}$.
\end{theorem}

\begin{proof}
Without loss of generality, we assume that $p \leq q$.  Consider the subgraph of $\widetilde{K}_{p+q+1}$ induced by vertices $1$ through $p$, and let $\alpha =(\alpha_1,\alpha_2,...,\alpha_p)$.  Because we are dealing with a book representation, we know that there exists some edge $(\alpha_i, \alpha_{i+1})$ that is in a lower sheet than all other edges in the cycle.  Change the orientation of the cycle if necessary so that $\alpha_i<\alpha_{i+1}$.  Edges $(\alpha_i, p+q+1)$ and $(\alpha_{i+1},p+q)$ are also in lower sheets than any of the edges of $\alpha$, so the cycle $\tilde{\alpha}=(\alpha_1, ..., \alpha_i, p+q+1, p+q, \alpha_{i+1}, \alpha_{i+2}, ..., \alpha_p)$ is ambient isotopic to $\alpha$.

Similarly, we can find a cycle $(\beta_1, \beta_2, ... \beta_q)$ that is ambient isotopic to $\beta$ using vertices $p+1$ through $p+q$.  We know such a cycle exists, because the subgraph induced by any $q$ vertices is ambient isotopic to the canonical book representation of $K_q$.  Suppose that $\beta_j = p+q$, and that the cycle is oriented such that $\beta_{j-1}<\beta{j+1}$.  Using the same argument used in the proof of Theorem \ref{extendCycle}, we can extend $\beta$ to an ambient isotopic cycle $\tilde{\beta}=(\beta_1, \beta_2, ... \beta_{j-1}, p+q+1, p+q, \beta_{j+1}, ..., \beta_q)$ that contains the edge $(p+q, p+q+1)$.

The cycles $\tilde{\alpha}$ and $\tilde{\beta}$ meet along the edge $(p+q,p+q+1)$.  The only crossing between disjoint edges in the two cycles is a single crossing between the edges $(\alpha_{i+1},p+q)$ and $(\beta_{j-1},p+q+1)$.  Since this crossing can be eliminated by flipping one of the components $\tilde{\alpha}$ or $\tilde{\beta}$, the cycle
\[ (\alpha_1, ... \alpha_i, p+q+1, \beta_{j-1}, \beta_{j-2}, ..., \beta_{1},\beta_{q},\beta_{q-1},...\beta_{j+1}, p+q, \alpha_{i+1}, ..., \alpha_p) \]
is ambient isotopic to the composite knot $\alpha \# \beta$.
\end{proof}

\section{Conclusion}
\label{Conclusion}
Using a computer program that identifies knots from their Dowker-Thistlethwaite code \cite{toth}, we obtain the following counts for knotted Hamiltonian cycles in the canonical book embedding of $K_n$ for $8\leq n \leq 11$:
\vspace{1 em}

\begin{tabular}{c|c|c}
$n$ & Knotted Hamiltonian cycles & Total number of knotted cycles \\
  \hline
  8 & 21 $3_1$ knots & 29 \\
  \hline
  9 & 342 $3_1$ knots & 577 \\
    & 9 $4_1$ knots & \\
    & 1 $5_1$ knot & \\
  \hline
  10 & 5090 $3_1$ knots & 9991 \\
     & 245 $4_1$ knots &  \\
     & 50 $5_1$ knots &   \\
     & 20 $5_2$ knots &   \\
     & 1 $8_{19}$ knot & \\
  \hline
  11 & 74855 $3_1$ knots & 165102 \\
     & 5335 $4_1$ knots &  \\
     & 1375 $5_1$ knots &   \\
     & 836 $5_2$ knots &   \\
     & 11 $6_1$ knots & \\
     & 11 $6_2$ knots & \\
     & 1 $7_1$ knot & \\
     & 56 $8_{19}$ knot & \\
     & 1 $10_{124}$ knot & \\
  \hline
\end{tabular}
\vspace{1 em}

The values in column 3 are a consequence of the following:

\begin{proposition}
Let $f(n)$ be the number of knotted Hamiltonian cycles in the canonical book representation of $K_n$.  Then the total number of knotted cycles in $\widetilde{K}_n$ is
\[ \sum_{j=7}^{n} \binom{n}{j} f(j) \]
\end{proposition}
\begin{proof}
The proof follows immediately from Otsuki's result that any subset of vertices induces a subgraph that is ambient isotopic to the canonical book representation.
\end{proof}

In \cite{hirano}, Hirano proves that all spatial embeddings of $K_8$ must have at least 3 knotted Hamiltonian cycles; however, no known example achieves that bound.  In Proposition 18 of \cite{abrams}, Abrams and Mellor proved that the minimum number of knotted cycles in $K_8$ must be between 15 and 29.  We conjecture the following:

\begin{guess}The canonical book representation of $K_n$ contains the fewest total number of knotted cycles possible in any embedding of $K_n$.
\label{knottedCyclesConjecture}
\end{guess}

\begin{guess}
The canonical book representation of $K_n$ contains the fewest number of knotted Hamiltonian cycles possible in any embedding of $K_n$.
\label{HamiltonianConjecture}
\end{guess}

Note that Conjecture \ref{HamiltonianConjecture} implies Conjecture \ref{knottedCyclesConjecture}.  Both conjectures are true for $n\leq 7$.

\section{Acknowledgments}
The authors gratefully acknowledge funding for this project received from the Merrimack College Paul E. Murray Fellowship.  We also thank David Toth and Michael Walton for helpful conversations and for their assistance writing code to count the number of knotted Hamiltonian cycles in an embedding.

\bibliographystyle{plain}
\bibliography{sourcescited}

\begin{thebibliography}{10}

\bibitem{abrams}
Loren Abrams and Blake Mellor.
\newblock Counting links and knots in complete graphs.
\newblock arXiv:math.GT/1008.1085, August 2010.

\bibitem{adams}
Colin~C. Adams.
\newblock {\em The knot book}.
\newblock American Mathematical Society, Providence, RI, 2004.
\newblock An elementary introduction to the mathematical theory of knots,
  Revised reprint of the 1994 original.

\bibitem{Bernhart}
Frank Bernhart and Paul~C. Kainen.
\newblock The book thickness of a graph.
\newblock {\em J. Combin. Theory Ser. B}, 27(3):320--331, 1979.

\bibitem{BraidsRef_2}
Joan~S. Birman.
\newblock {\em Braids, links, and mapping class groups}.
\newblock Princeton University Press, Princeton, N.J., 1974.
\newblock Annals of Mathematics Studies, No. 82.

\bibitem{CG}
J.~H. Conway and C.~McA. Gordon.
\newblock Knots and links in spatial graphs.
\newblock {\em J. Graph Theory}, 7(4):445--453, 1983.

\bibitem{fleming}
Thomas Fleming and Blake Mellor.
\newblock Counting links in complete graphs.
\newblock {\em Osaka J. Math.}, 46(1):173--201, 2009.

\bibitem{hirano}
Yoshiyasu Hirano.
\newblock Improved lower bound for the number of knotted {H}amiltonian cycles
  in spatial embeddings of complete graphs.
\newblock {\em J. Knot Theory Ramifications}, 19(5):705--708, 2010.

\bibitem{Kobayashi}
Kazuaki Kobayashi.
\newblock Standard spatial graph.
\newblock {\em Hokkaido Math. J.}, 21(1):117--140, 1992.

\bibitem{otsuki}
Takashi Otsuki.
\newblock Knots and links in certain spatial complete graphs.
\newblock {\em J. Combin. Theory Ser. B}, 68(1):23--35, 1996.

\bibitem{toth}
David Toth and Michael Walton.
\newblock Personal communication, July 2010.

\end{thebibliography}

\end{document}